\documentclass{amsart}

\usepackage[all]{xy}

\newcommand{\essinf}{\inf^*\!}

\newcommand\w{\omega}
\newcommand{\IN}{\mathbb N}
\newcommand{\cl}{\mathrm{cl}}

\newtheorem{theorem}{Theorem}
\newtheorem{problem}{Problem}
\newtheorem{lemma}{Lemma}
\newtheorem{corollary}{Corollary}
\newtheorem{claim}{Claim}

\theoremstyle{definition}

\newtheorem{remark}{Remark}

\title{On images of complete topologized subsemilattices in sequential  semitopological  semilattices}
\author{Taras Banakh, Serhi\u\i\ Bardyla}
\address{T.Banakh: Ivan Franko National University of Lviv (Ukraine) and Jan Kochanowski University in Kielce (Poland)}
\email{t.o.banakh@gmail.com}
\address{S.~Bardyla: Institute of Mathematics, Kurt G\"{o}del Research Center, Vienna (Austria)}
\thanks{The second author was supported by the Austrian Science Fund FWF (Grant  I 3709-N35).}
\email{sbardyla@yahoo.com}
\begin{document}
\begin{abstract} A topologized semilattice $X$ is called {\em complete} if each non-empty chain $C\subset X$ has $\inf C\in\bar C$ and $\sup C\in\bar C$. 
We prove that for any continuous homomorphism $h:X\to Y$ from a complete topologized semilattice $X$  to a sequential Hausdorff semitopological semilattice $Y$ the image $h(X)$ is closed in $Y$.
\end{abstract}
\maketitle

This paper is a continuation of the investigations \cite{BBm}, \cite{BBc}, \cite{BBw} of complete topologized semilattices.

A {\em semilattice} is any commutative semigroup of idempotents (an element $x$ of a semigroup is called an {\em idempotent} if $xx=x$). 

A semilattice endowed with a topology is called a {\em topologized semilattice}. 
A topologized semilattice $X$ is called a ({\em semi}){\em topological semilattice} if the semigroup operation $X\times  X\to X$, $(x,y)\mapsto xy$, is (separately) continuous.

Each semilattice carries a natural partial order $\le$ defined by $x\le y$ iff $xy=x=yx$. Endowed with this partial order, the semilattice is a {\em poset}, i.e., partially ordered set. 
Many properties of a semilattice can be expressed in the language of this partial order. In particular, a subset $C$ of a semilattice (more generally, poset) $X$ is called a {\em chain} if any points $x,y\in C$ are comparable in the sense that $x\le y$ or $y\le x$. A poset $X$ is called {\em chain-finite} if each chain in $X$ is finite.


 

In \cite{Stepp1975} Stepp proved that for any homomorphism $h:X\to Y$ from a chain-finite semilattice to a Hausdorff topological semilattice $Y$ the image $h(X)$ is closed in $Y$. In \cite{BBm}, the authors improved this result of Stepp proving the following theorem.

\begin{theorem}\label{t:cf}  For any homomorphism $h:X\to Y$ from a chain-finite semilattice to a Hausdorff semitopological semilattice $Y$, the image $h(X)$ is closed in $Y$.
\end{theorem} 

A topological counterpart of the notion of a chain-finite poset is the notion of a complete topologized poset. A {\em topologized poset} is a poset $(X,\le)$ endowed with a topology. A topologized poset $X$ is called {\em complete} if each chain $C\subset X$ has $\inf C$ and $\sup C$ that belong to the closure $\bar C$ of the chain $C$ in $X$.

Complete topologized semilattices were introduced in \cite{BBm} under the name {\em $k$-complete} topologized semilattices. But we prefer to call such topologized semilattice complete (taking into account the fundamental role of complete topologized semilattice in the theory of absolutely closed topologized semilattices, see  \cite{BBm}, \cite{BBc}, \cite{BBw}, \cite{BBo}). 
In \cite{BBm} the authors proved the following closedness property of complete topologized semilattices.

\begin{theorem}\label{t:kc} For any continuous homomorphism $h:X\to Y$ from a complete topologized semilattice $X$ to a Hausdorff topological semilattice $Y$ the image $h(X)$ is closed in $Y$.
\end{theorem}

Theorems~\ref{t:cf} and \ref{t:kc} motivate the following (still) open problem.

\begin{problem}\label{prob:main} Assume that $h:X\to Y$ is a continuous homomorphism from a complete topologized semilattice $X$ to a Hausdorff semitopological semilattice $Y$. Is $h(X)$ closed in $Y$?
\end{problem}
  
In this paper we answer this problem affirmatively under the additional condition that the semitopological semilattice $Y$ is sequential. We recall that a topological space $Y$ is {\em sequential} if each sequentially closed subset of $Y$ is closed.
A subset $A\subset Y$ is called {\em sequentially closed} if $A$ contains the limit points of all sequences $\{a_n\}_{n\in\w}\subset A$ that converge in $Y$. 

The following theorem is the main result of this paper. 

\begin{theorem}\label{t:main} For any continuous homomorphism $h:X\to Y$ from a complete topologized semilattice $X$ to a sequential Hausdorff semitopological semilattice $Y$ the image $h(X)$ is closed in $Y$.
\end{theorem}

Theorem~\ref{t:main} will be proved in Section~\ref{s:im} after some preliminary work made in Sections~\ref{s:pre}, \ref{s:sub}. More precisely, Theorem~\ref{t:main} is an immediate corollary of Corollary~\ref{c:im-t} treating homomorphisms from complete topologized semilattices to Hausdorff semitopological semilattices of countable tightness.

\section{Some properties of ${\updownarrow}\w$-complete topologized semilattices}\label{s:pre}

In this section we shall prove some properties of complete topologized semilattices, which will be used in the proof of our main results.

First we introduce a parametric version of completeness. Let $\kappa$ be an infinite  cardinal. A topologized semilattice $X$ is defined to be
\begin{itemize}
\item {\em ${\uparrow}\kappa$-complete} if any non-empty chain $C\subset X$ of cardinality $|C|\le\kappa$ has $\sup C\in\bar C$;
\item {\em ${\downarrow}\kappa$-complete} if any non-empty chain $C\subset X$ of cardinality $|C|\le\kappa$ has $\inf C\in\bar C$;
\item {\em ${\updownarrow}\kappa$-complete} if $X$ is ${\uparrow}\kappa$-complete and ${\downarrow}\kappa$-complete;
\item {\em down-complete} if $X$ is ${\downarrow}\kappa$-complete for any cardinal $\kappa$.
\end{itemize}
It is clear that a topologized semilattice $X$ is complete if and only if $X$ is ${\updownarrow}\kappa$-complete for every cardinal $\kappa$ if and only if 
$X$ is ${\updownarrow}\kappa$-complete for the cardinal $\kappa=|X|$.

A subset $D$ of a poset $(X,\le)$ is called {\em up-directed} (resp. {\em down-directed\/}) if for any elements $x,y\in D$ there exists an element $z\in D$ such that $x\le z$ and $y\le z$ (resp. $z\le x$ and $z\le y$). It is clear that each chain is both up-directed and down-directed.


\begin{lemma}\label{l:up-p} If a topologized semilattice $X$ is ${\uparrow}\w$-complete, then any non-empty countable up-directed subset $D\subset X$ has $\sup D\in\bar D$.
\end{lemma}

\begin{proof} Assume that the topologized semilattice $X$ is ${\uparrow}\w$-complete and take any non-empty countable directed subset $D=\{x_n\}_{n\in\w}$ in $X$. Put $y_0:=x_0$ an for every $n\in\IN$ choose an element $y_n\in D$ such that $y_n\ge x_n$ and $y_n\ge y_{n-1}$ (such an element $y_n$ exists as $D$ is directed).

By the ${\uparrow}\w$-completeness, the chain $C:=\{y_n\}_{n\in\w}$ has $\sup C\in \bar C\subset\bar D$. We claim that $\sup C$ is the largest lower bound of the set $D$. Indeed, for any $n\in\w$ we get $x_n\le y_n\le\sup D$ and hence $\sup C$ is an upper bound for the set $D$. On the other hand, each upper bound $b$ for $D$ is an upper bound for $C$ and hence $\sup C\le b$. Therefore $\sup D=\sup C\in\bar C\subset\bar D$.
\end{proof}


\begin{lemma}\label{l:down} If a topologized semilattice $X$ is ${\downarrow}\w$-complete, then each non-empty countable subset $A\subset X$ has $\inf A\in \bar S$ where $S$ is the semilattice generated by $A$ in $X$.
\end{lemma} 

\begin{proof} Let $A=\{x_n\}_{n\in\w}$ be a countable set in $X$. By the ${\downarrow}\w$-completeness of $X$, the chain $C:=\{x_0\cdots x_n\}_{n\in\w}\subset S$ has $\inf C\in\bar C\subset\bar S$. Taking into account that $\inf A=\inf C$, we conclude that $\inf A=\inf C\in\bar S$.
\end{proof}


Let $X$ be an ${\updownarrow}\w$-complete topologized semilattice, $A$ is a non-empty countable set in $X$ and $S$ be a subsemilattice of $X$, generated by $A$.
Let $[A]^{<\w}$ be the family of all finite subsets of $A$. 
Lemma~\ref{l:down} implies that for each $F\in[A]^{<\w}$ the set $A\setminus F$ has $\inf (A\setminus F)\in\bar S$. By Lemma~\ref{l:up-p}, the directed set $D:=\{\inf (A\setminus F):F\in[A]^{<\w}\}$ has $\sup D\in\bar D\subset \bar S$. The element $$\mathrm{inf}^*\! A:=\sup_{F\in[A]^{<\w}}\inf(A\setminus F)=\sup D\in\bar S$$ will be called the {\em essential infimum} of the set $A$ in $X$. 

An important property of the essential infimum is that $\inf^* B\le\inf^* A$ for any countable subset $A\subset^* B$ in $X$. Here the symbol $A\subset^* B$ means that the complement $A\setminus B$ is finite, so $A$ is almost included into $B$.

We shall say that an infinite subset $A$ of a topological space $X$ {\em converges to a point} $x\in X$ if  each neighborhood $O_x\subset X$ of $x$ contains all but finitely many points of the set $A$ (which means that $A\subset^* O_x$). If an infinite set $A$ converges to a point $x$, then any infinite subset $B\subset A$ also converges to $x$.

We are going to show that for any ${\updownarrow}\w$-complete Hausdorff topologized semilattice $X$ containing no strictly increasing transfinite sequences of length $\mathfrak t$, any countable subset $A\subset X$ that converges to a point $x$ contains an infinite subset $B\subset A$ such that $\inf^*B=x$.
\smallskip

The cardinal $\mathfrak t$ (called the {\em tower number} in \cite{vD}) is defined as the smallest cardinal $\kappa$ for which there exists a transfinite sequence $(T_\alpha)_{\alpha\in \kappa}$ of infinite subsets of $\w$ having the following two properties:
\begin{enumerate}
\item $T_\beta\subset^* T_\alpha$ for all $\alpha<\beta<\kappa$;
\item for any infinite set $I\subset \w$ there exists $\alpha\in\kappa$ such that $I\not\subset^* T_\alpha$.
\end{enumerate}
It is known \cite{vD}, \cite{Vau} that $\w_1\le\mathfrak t\le\mathfrak c$, and $\mathfrak t=\mathfrak c$ under Martin's Axiom. By a recent breakthrough result of Malliaris and Shelah  \cite{MS}, $\mathfrak t$ is equal to the {\em pseudointersection number} $\mathfrak p$, defined as the smallest cardinality of a non-empty family $\mathcal A$ of subsets of $\w$ such that for any finite subfamily $\mathcal F\subset\mathcal A$ the intersection $\cap\mathcal F$ is infinite but for any infinite set $I\subset \w$ there exists $A\in\mathcal A$ such that $I\not\subset^* A$. More information on cardinals $\mathfrak p$, $\mathfrak t$ and other cardinal characteristics of the continuum can be found in the surveys \cite{vD}, \cite{Vau}, \cite{Blass}. We identify cardinals with the smallest ordinals of a given cardinality.

Let $\kappa$ be an ordinal. A transfinite sequence $(x_\alpha)_{\alpha\in\kappa}$ of points of a partially ordered set $X$ is called {\em strictly increasing} if $x_\alpha<x_\beta$ for any ordinals $\alpha<\beta$ in $\kappa$.

\begin{lemma}\label{l:t2} Let $Y$ be a Hausdorff semitopological semilattice and $X$ be an ${\updownarrow}\w$-complete subsemilattice of $Y$ containing no strictly increasing transfinite sequences of length $\mathfrak t$. Then each infinite set $A\subset X$ that converges to a point $y\in Y$ contains an infinite subset $B\subset A$ such that $\essinf B=y$.
\end{lemma}

\begin{proof} Fix an infinite set $A\subset X$ that converges to a point $y\in Y$. To derive a contradiction, assume that $\essinf B\ne y$ for any infinite subset $B\subset A$. 

\begin{claim}\label{cl:new2} For any infinite subset $B\subset A$ there exists an infinite subset $C\subset B$ such that $\essinf B<\essinf C$.
\end{claim}

\begin{proof} By our assumption, $\essinf B\ne y$. By the Hausdorff property of $Y$, there exists an open neighborhood $U\subset Y$ of $y$ such that $\essinf B\notin \bar U$. 

Inductively we shall construct a sequence of pairwise distinct points $\{x_n\}_{n\in\w}$ in $B$ such that for every $k\le n$ the product $y_{k,n}:=x_k\cdots x_ny$ is contained in $U$. 

To start the inductive construction, find a neighborhood $V_{0}\subset U$ of $y$ such that $V_{0}y\subset U$ and choose any point $x_0\in V_{0}\cap B$ (such a point exists as $B\subset A$ converges to $y$). It follows that $x_0y\in V_0y\subset U$.

Assume that for some $n\in\w$  points $x_0,\dots,x_n\in B$ are chosen so that $y_{k,n}:=x_k\cdots x_ny\in U$ for every $k\le n$. 
For every $k\le n$ choose a neighborhood $V_{k,n}\subset U$ of $y$ such that $y_{k,n}V_{k,n}y\subset U$ (such neighborhood exists since $y_{k,n}yy=y_{k,n}\in U$ and $Y$ is a semitopological semilattice). Consider the neighborhood $V_{n+1}:=\bigcap_{k\le n}V_{k,n}$ of $y$ and choose any point $x_{n+1}\in V_{n+1}\cap B\setminus\{x_0,\dots,x_n\}$. For every $k\le n$ the choice of the neighborhood $V_{k,n}$ guarantees that $y_{k,n+1}:=x_k\cdots x_nx_{n+1}y\in U$. This completes the inductive step.
\smallskip

After completing the inductive construction, consider the set $C=\{x_k\}_{k\in\w}\subset B$. We claim that $\inf^*C=y\cdot\inf^* C\in \bar U$. 
For every $k\in\w$ consider the set $C_k=\{x_n\}_{n\ge k}\subset C$ and let $S_k$ be the subsemilattice generated by $C_k$ in $X$. By the ${\downarrow}\w$-completeness of $X$ and Lemma~\ref{l:down}, the set $C_k$ has $\inf C_k=\inf S_k\in\bar S_k$. In fact, $\inf C_k=\inf S_k=\inf\{x_{k}\cdots x_n:n\ge k\}\in\cl_X(\{x_k\cdots x_n:n\ge k\})$. Observe that $\inf C_k\le x_n$ for all $n\ge k$. Consequently, $y\in \cl_Y(\{x_n\}_{n\ge k})\subset{\uparrow}\inf C_k$ and hence $y\cdot\inf C_k=\inf C_k$.

The continuity of the shift $s_y:Y\to Y$, $s_y:z\mapsto zy$, guarantees that $\inf C_k=y\cdot\inf C_k\in y\cdot\cl_X(\{x_k\cdots x_n:n\ge k\})\subset \cl_Y(\{x_k\cdots x_ny:n\ge k\})\subset \bar U$.

By the ${\uparrow}\w$-completeness of $X$, we have $\essinf C=\sup_{k\in\w}\inf C_k\in \cl_X(\{\inf C_k:k\in\w\})\subset \bar U$ and hence $\essinf C\ne\essinf B$ as $\essinf B\notin \bar U$. Taking into account that $\essinf B\le \essinf C$, we conclude that $\essinf B<\essinf C$.
\end{proof}

Choose any countable infinite subset $A_0\subset A$. 
Claim~\ref{cl:new2} (on successor steps) and the definition of the cardinal $\mathfrak t$ (on limit steps) help us to construct a transfinite sequence $(A_\alpha)_{\alpha<\mathfrak t}$ of infinite subsets of $A_0$ such that for any $\alpha<\mathfrak t$ the following conditions are satisfied:
\begin{enumerate}
\item $A_{\alpha+1}\subset A_\alpha$ and $\essinf A_\alpha<\essinf A_{\alpha+1}$.
\item $A_\alpha\subset^* A_\gamma$ for all $\gamma<\alpha$.
\end{enumerate}
Then $\big(\essinf A_\alpha\big){}_{\alpha\in\mathfrak t}$ is a strictly increasing transfinite sequence of length $\mathfrak t$ in $X$, whose existence is forbidden by our assumption. This contradiction completes the proof of the lemma.
\end{proof}

We recall that a topological space $X$ is defined to have {\em countable tightness} if for any set $A\subset X$ and point $a\in\bar A$, there exists a countable set $B\subset A$ such that $a\in\bar B$. It is known \cite[1.7.3{c}]{Engelking1989} that each subspace of a sequential space has countable tightness.

\begin{lemma}\label{l:t} A topologized semilattice $X$ contains no strictly increasing transfinite sequences of length $\w_1$ if $X$ is ${\uparrow}\w_1$-complete, the space $X$ has countable tightness, and for every $x\in X$ the lower set ${\downarrow}x$ is closed in $X$.
\end{lemma}

\begin{proof} To derive a contradiction, assume that $X$ contains a strictly increasing transfinite sequence $(x_\alpha)_{\alpha\in\w_1}$. By the ${\uparrow}{\w_1}$-completeness of $X$ the chain $C=\{x_\alpha:\alpha\in\w_1\}$ has $c:=\sup C\in \bar C$. Since $(x_\alpha)_{\alpha\in\w_1}$ is strictly increasing, $c\notin C$. By the countable tightness of $X$, there exists a countable set $B\subset C$ such that $c\in\bar B$. By the countability of $B$, there exists a countable ordinal $\beta$ such that  $B\subset\{x_\alpha:\alpha<\beta\}$. Then $c\in\bar B\subset\overline{{\downarrow}x_\beta}={\downarrow}x_\beta$ and hence $c\le x_\beta<x_{\beta+1}\le c$, which is a desired contradiction.
\end{proof}

\section{The closedness of complete subsemilattices}\label{s:sub}

In this section we search for conditions of (sequential) closedness of a ${\updownarrow}\w$-complete semilattice $X$ in a Hausdorff semitopological semilattice $Y$.

\begin{theorem}\label{t:t} Let $Y$ be a Hausdorff semitopological semilattice and $X$ be a ${\updownarrow}\w$-complete subsemilattice of $Y$. If $X$ contains no strictly increasing transfinite sequences of length $\mathfrak t$, then $X$ is sequentially closed in $Y$ and the partial order $\{(x,y)\in X\times X:xy=x\}$ of $X$ is sequentially closed in $Y\times Y$.
\end{theorem}

\begin{proof} Assuming that $X$ is not sequentially closed in $Y$, we can find a sequence $(x_n)_{n\in\w}$ of pairwise distinct points of $X$ that converges to a point $y\in Y\setminus X$. It follows that the infinite set $A=\{x_n\}_{n\in\w}\subset X$ converges to $y$. By Lemma~\ref{l:t2}, the set $A$ contains an infinite subset $B\subset A$ such that $y=\inf^* B\in X$. But this contradicts the choice of $y$.
\smallskip

To show that the partial order $P=\{(x,y)\in X\times X:xy=x\}$ of $X$ is sequentially closed in $Y\times Y$, fix any sequence $\{(x_n,y_n)\big\}_{n\in\w}\subset P$ that converges to a pair  $(x,y)\in Y\times Y$. We should prove that $(x,y)\in P$. Since $X$ is sequentially closed in $Y$, the limits $x,y$ of the sequences $(x_n)_{n\in\w}$ and $(y_n)_{n\in\w}$ belong to $X$. The separate continuity of the semigoup operation on the Hausdorff space $X$ implies that the sets ${\uparrow}x:=\{z\in Y:zx=x\}$ and ${\downarrow}y:=\{z\in Y:zy=z\}$ are closed in $Y$.

If the set $\{x_n\}_{n\in\w}$ is finite, then the convergence $x=\lim_{n\to\infty}x_n$ implies that the set $I:=\{n\in \w:x_n=x\}$ is infinite and then $y\in\cl_Y(\{y_n\}_{n\in I})\subset {\uparrow}x$ (as $x=x_n\le y_n$ for all $n\in I$) and hence $(x,y)\in P$. So, we can assume that the set $\{x_n\}_{n\in\w}$ is infinite. Applying Lemma~\ref{l:t2}, we can find an infinite set $\Omega\subset\w$ such that $x=\inf^*\{x_n\}_{n\in\Omega}$. 

If the set $\{y_n\}_{n\in\Omega}$ is finite, then the convergence $y=\lim_{n\to\infty}y_n$ implies that the set $J:=\{n\in \Omega:y_n=y\}$ is infinite and then $x\in\cl_X(\{x_n\}_{n\in J})\subset {\downarrow}y$ and $(x,y)\in P$.

So we can assume that the set $\{y_n\}_{n\in\Omega}$ is infinite. Applying Lemma~\ref{l:t2}, we can find an infinite set $\Lambda\subset\Omega$ such that $y=\inf^*\{y_n\}_{n\in\Lambda}$. Taking into account that $x_n\le y_n$ for all $n\in\w$, we see that $x=\inf^*\{x_n\}_{n\in\Omega}\le\inf^*\{x_n\}_{n\in\Lambda}\le \inf^*\{y_n\}_{n\in\Lambda}=y$ and hence $(x,y)\in P$.
\end{proof}

Theorem~\ref{t:t} and Lemma~\ref{l:t} imply

\begin{corollary}\label{c:t} Let $X$ be a subsemilattice of a Hausdorff semitopological semilattice $Y$ such that $X$ is ${\downarrow}\w$-complete and ${\uparrow}\w_1$-complete. If $X$ has countable tightness,  then $X$ is sequentially closed in $Y$ and the partial order $\{(x,y)\in X\times X:xy=x\}$ of $X$ is sequentially closed in $Y\times Y$.
\end{corollary}

Taking into account that subspaces of sequential spaces have countable tightness \cite[1.7.3{c}]{Engelking1989}, we can see that Corollary~\ref{c:t} implies another corollary.

\begin{corollary}\label{c2} Let $X$ be subsemilattice of a sequential Hausdorff semitopological semilattice $Y$. If $X$ is ${\downarrow}\w$-complete and ${\uparrow}\w_1$-complete, then $X$ is closed in $Y$ and the partial order $\{(x,y)\in X\times X:xy=x\}$ of $X$ is sequentially closed in $Y\times Y$.
\end{corollary}

\begin{remark} The ${\downarrow}\w$-completeness of $X$ is essential in Corollaries~\ref{c:t} and \ref{c2}: by \cite{BBR2}, there exists a metrizable semitopological semilattice $X$ whose partial order is not closed in $X\times X$, and for every $x\in X$ the upper set ${\uparrow}x$ is finite.
\end{remark}

\section{The closedness of images of complete semilattices}\label{s:im}

In this section we apply the results of the preceding section to establish the (sequential) closedness of images of complete semilattices under continuous homomorphisms.

\begin{lemma}\label{l:im} Let $h:X\to Y$ be a surjective continuous homomorphism from a down-complete topologized semilattice $X$ to a Hausdorff semitopological  semilattice $Y$. Let  $\kappa$ be a cardinal. 
\begin{enumerate}
\item The topologized semilattice $Y$ is a down-complete.
\item If $X$ is ${\uparrow}\kappa$-complete, then $Y$ is a ${\uparrow}\kappa$-complete.
\item If $X$ contains no strictly increasing transfinite sequences of length $\kappa$, then $Y$ contains no strictly increasing transfinite sequences of length $\kappa$.
\end{enumerate}
\end{lemma}

\begin{proof} For every $y\in Y$ consider the closed subsemilattice $S_y:=h^{-1}({\uparrow}y)$ in $X$. Let $M_y$ be a maximal chain in $S_y$. By the down-completeness of $X$, the chain $M_y$ has $\inf M_y\in\overline{M}_y\subset \overline S_y=S_y$. We claim that $s_y:=\inf M_y$ is the smallest element of $S_y$. In the opposite case, there would exist an element $x\in S_y$ such that $s_y\not\le x$ and hence $xs_y<s_y$. Then $\{xs_y\}\cup M_y$ is a chain in $S_y$, properly containing the maximal chain $M_y$, which is a desired contradiction showing that $s_y$ is the smallest element $\min S_y$ of the semilattice $S_y$. It follows from $h(S_y)={\uparrow}y$ that $h(s_y)\in h(S_y)={\uparrow}y$ and hence $y\le h(s_y)$. On the other hand, for any $x\in h^{-1}(y)$ we get $s_y\le x$ and hence $h(s_y)\le h(x)=y$ and finally $h(s_y)=y$. It is clear that for any $x\le y$ in $C$, we get $S_x\subset S_y$ and hence $\min S_x\le\min S_y$.
\smallskip 

1. To prove that $Y=h(X)$ is down-complete, we should show that  any non-empty chain $C\subset Y$ has $\inf C\in \bar C$. It follows that for any $x\le y$ in $C$, we get $\min S_x\le\min S_y$, which means that $D:=\{\min S_x\}_{x\in C}$ is a chain in $X$. By the down-compactness of $X$, the chain $D$ has $\inf D\in\bar D$.

The continuity of the homomorphism $h$ ensures that $h(\inf D)\in h(\overline D)\subset \overline{h(D)}=\overline{C}$.  It remains to check that $h(\inf D)=\inf C$.
Taking into account that $h$ is a semilattice  homomorphism, we can show that $h(\inf D)$ is a lower bound of the set $C=h(D)$ in $Y=h(X)$. For any other lower bound $b\in h(X)$ of $C$, we see that $C\subset{\uparrow} b$, $D\subset h^{-1}({\uparrow}b)=S_b$ and hence $\min S_b\le \inf D$, which implies that $b=h(\min S_b)\le h(\inf D)$. So, $\inf C=h(\inf D)\in\bar C$. 
\smallskip

2. Assuming that $X$ is ${\uparrow}\kappa$-complete, we shall check that the topologized semilattice $Y$ is ${\uparrow}\kappa$-compact. Given any non-empty chain $C\subset Y$ of cardinality $|C|\le\kappa$, we should show that $C$ has $\sup C\in \bar C$. It is clear that for any $x\le y$ in $C$, we get $\min S_x\le\min S_x$, which means that $D:=\{\min S_x\}_{x\in C}$ is a chain in $X$ of cardinality $|D|\le|C|\le\kappa$. Since $X$ is ${\uparrow}\kappa$-complete, the chain $D$ has $\sup D\in\bar D$.

The continuity of the homomorphism $h$ ensures that $h(\sup D)\in h(\overline D)\subset \overline{h(D)}=\overline{C}$.  It remains to check that $h(\sup D)=\sup C$.
Taking into account that $h$ is a semilattice  homomorphism, we can show that $h(\sup D)$ is an upper bound for the set $C=h(D)$ in $Y=h(X)$. For any other upper bound $b\in h(X)$ of $b$, we see that $C\subset{\downarrow} b$, $D\subset{\downarrow}\min S_b$ and hence $\sup D\le \min S_b$, which implies that  $h(\sup D)\le h(\min S_b)=b$. So, $\sup C=h(\sup D)\in\bar C$. 
\smallskip

3. Assuming that $(y_\alpha)_{\alpha\in\kappa}$ is a strictly increasing transfinite sequence of length $\kappa$ in $Y$, we can see that $(\min S_{y_\alpha})_{\alpha\in\kappa}$ is a strictly increasing transfinite sequence of length $\kappa$ in $X$.
\end{proof}

Lemma~\ref{l:im} and Theorem~\ref{t:t} imply:

\begin{corollary} Let $X$ be a down-complete ${\uparrow}\w$-complete topologized semilattice containing no strictly increasing transfinite sequences of length $\mathfrak t$.
For any continuous homomorphism $h:X\to Y$ to a Hausdorff semitopological semilattice $Y$, the image $h(X)$ is sequentially closed in $Y$.
\end{corollary}

Lemma~\ref{l:im} and Corollary~\ref{c:t} imply:

\begin{theorem}\label{t:im-t} Let $X$ be a countably tight down-complete ${\uparrow}\w_1$-complete topologized semilattice. For any continuous homomorphism $h:X\to Y$ to a Hausdorff semitopological semilattice $Y$, the image $h(X)$ is sequentially closed in $Y$.
\end{theorem}

\begin{corollary}\label{c:im-t} For any continuous homomorphism $h:X\to Y$ from a  down-complete ${\uparrow}\w_1$-complete topologized semilattice $X$ to a Hausdorff semitopological semilattice $Y$ of countable tightness, the image $h(X)$ is sequentially closed in $Y$.
\end{corollary}

Since  sequential spaces have countable tightness \cite[1.7.13(c)]{Engelking1989}, Corollary~\ref{c:im-t} implies Theorem~\ref{t:main} announced in the introduction.

\end{document}